\documentclass[psamsfonts]{amsart}

\usepackage{amssymb}
\usepackage{amsmath}
\usepackage{amsthm}
\usepackage[urlcolor=black, colorlinks=true, citecolor=black, linkcolor=black]{hyperref}
\usepackage[english]{babel}

\usepackage{dsfont} 

\newcommand{ \N } { \mathbb{N} }

\newcommand{\w}{\omega}
\newcommand{\wstar}{\omega^*}
\newcommand{\Cstar}{C^*}
\newcommand{\cont}{\mathfrak{c}}

\newcommand{\closure}[1]{\overline{#1}}

\newcommand{\Set}[1]{{\left\lbrace {#1} \right\rbrace}}
\newcommand{\singleton}{\Set}

\def\set#1:#2{\Set{{#1} \colon {#2}}}

\begin{document}
\title{Finite compactifications of $\wstar \setminus \singleton{x}$}
\author{Max F.\ Pitz}
\author{Rolf Suabedissen}
\address{Mathematical Institute\\University of Oxford\\Oxford OX2 6GG\\United Kingdom}
\email[Corresponding author]{pitz@maths.ox.ac.uk}
\address{Mathematical Institute\\University of Oxford\\Oxford OX2 6GG\\United Kingdom}
\email{suabedis@maths.ox.ac.uk}

\begin{abstract}
We prove that under [CH], finite compactifications of $\wstar \setminus \singleton{x}$ are homeomorphic to $\wstar$. Moreover, in each case, the remainder consists almost exclusively of $P$-points, apart from possibly one point. 

Similar results are obtained for other, related classes of spaces, amongst them $S_\kappa$, the $\kappa$-Parovi\v{c}enko space of weight $\kappa$. Also, some parallels are drawn to the Cantor set and the Double Arrow space.
\end{abstract}

\maketitle
\thispagestyle{plain}

\newtheorem{mythm}{Theorem} \numberwithin{mythm}{section} 
\newtheorem{myprop}[mythm]{Proposition}
\newtheorem{mycor}[mythm]{Corollary}
\newtheorem{mylem}[mythm]{Lemma} 
\newtheorem{myquest}[mythm]{Question} 
\newtheorem{myconj}{Conjecture}

\section{Introduction}

Compactifications of $\wstar \setminus \singleton{x}$ have received considerable attention in the literature. Fine and Gillman \cite{fine} showed that under [CH], $\wstar$ is not the Stone-\v{C}ech compactification of any of its subspaces $\wstar \setminus \singleton{x}$, whereas van Douwen, Kunen and van Mill showed later that it is consistent with the usual axioms of set theory ZFC that this situation always occurs \cite{douwenkunenmill}. 

In \cite{pitz} the present authors obtained further structural results about the Stone-\v{C}ech compactification of $\wstar \setminus \singleton{x}$ under [CH]. Under the assumption $\kappa = \kappa^{<\kappa}$ we also extended our results to $S_\kappa$, the Stone space of the unique $\kappa$-saturated Boolean algebra of size $\kappa$. These spaces can be seen as the natural generalisation of $\wstar$ under [CH] to infinite cardinals $\kappa$ with the property $\kappa = \kappa^{<\kappa}$.

Yet, despite these far-reaching results regarding Stone-\v{C}ech compactifications of $\wstar \setminus \singleton{x}$,  to our knowledge nothing has appeared in print about finite compactifications of these spaces. In this paper we prove the somewhat surprising result that under [CH] every finite compactification of $\wstar \setminus \singleton{x}$ is again homeomorphic to $\wstar$. Moreover, each remainder consists almost exclusively of $P$-points, apart from possibly one point. This is accompanied by the result that there are arbitrarily large finite compactifications of $\wstar \setminus \singleton{x}$.

Corresponding results also hold for spaces $S_\kappa \setminus \singleton{x}$. A strong emphasis is laid on Lemma \ref{nicelemma}, which can serve as a roadmap of how to generalise statements from $\wstar$ to $S_\kappa$ that rely on the fact that $\wstar$ does not contain converging sequences.

We start our investigation in Section \ref{section2} with a general framework for proving that all finite compactifications of a given space look alike, illustrated by means of the Cantor set, the Double Arrow space and other related spaces. We then proceed in Sections \ref{section3} and \ref{section4} to $\wstar$ (under [CH]) and to $S_\kappa$ (for cardinals $\kappa$ such that $\kappa = \kappa^{<\kappa}$), and prove, after a proper introduction of these spaces, the announced results. 

The case $S_\kappa$ draws together the previous cases as $S_\w$ is homeomorphic to the Cantor set $C$, and $S_{\w_1}$ is homeomorphic to $\wstar$ under [CH].

\section{Theory and first examples}
\label{section2}

We begin with a sufficient condition for zero-dimensional locally compact Hausdorff spaces to have only one homeomorphism type amongst their finite compactifications. Recall that a space is \emph{zero-dimensional} if it has a basis of clopen (closed-and-open) sets.

\begin{mylem}
\label{theorem1}
Let $X$ be a zero-dimensional compact Hausdorff space such that $X \oplus X$ is homeomorphic to $X$ and
\begin{enumerate}
\item[$(\star)$]  for every point $x$ of $X$, every clopen non-compact subset $A$ of $X \setminus \singleton{x}$ is homeomorphic to $X \setminus \singleton{x_A}$ for some $x_A \in X$.
\end{enumerate}
Then, for all $x$, all finite compactifications of $X \setminus \singleton{x}$ are homeomorphic to $X$.
\end{mylem}



\begin{proof}
Let $Z$ be a finite compactification of $X \setminus \singleton{x}$ with remainder consisting of points $\infty_1,\ldots, \infty_n$. By \cite[2.3]{Woods}, every finite compactification of a locally compact zero-dimensional space is zero-dimensional. Hence, there is a partition of $Z$ into $n$ disjoint clopen sets $A_i$ such that $\infty_i \in A_i$. 

By property ($\star$), $A_i \setminus \singleton{\infty_i}$ is homeomorphic to $X \setminus \singleton{x_{A_i}}$. Uniqueness of the one-point compactification gives $A_i \cong X$ and hence $Z$ is, after applying $X \oplus X \cong X$ iteratively, homeomorphic to $X$.  
\end{proof}

This lemma lies at the heart of our main result regarding finite compactifications of $\wstar \setminus \singleton{x}$. Surprisingly, despite its strong assumptions, it also applies to a variety of other interesting spaces. 

Spaces which only have $\lambda$ different homeomorphism types amongst their open subspaces (for some cardinal $\lambda$) are said to be of \emph{diversity} $\lambda$ \cite{diversity}. One checks that Lemma \ref{theorem1} applies to all compact Hausdorff spaces of diversity two, which are known to be zero-dimensional \cite{zerodiv}. In particular, it applies to the Cantor space $C$, which can be characterised as the unique compact metrizable space of diversity two \cite{Cantor}, and to the Alexandroff Double Arrow space $D$ and also to their product $D \times C$. 

In a compact Hausdorff space $X$ of diversity two, any subspace $X \setminus \singleton{x}$ is homeomorphic to $X \setminus \Set{x_1,\ldots,x_n}$ and therefore has arbitrarily large finite compactifications. These are the cases where Lemma \ref{theorem1} is most valuable. The next lemma shows that not much is needed for this scenario to occur. The proof is a simple induction.
  
\begin{mylem}
\label{lemm2}
Let $X$ be a topological space such that for all $x$, all finite compactifications of $X \setminus \singleton{x}$ are homeomorphic to $X$. If all spaces $X \setminus \singleton{x}$ have two-point compactifications, they have arbitrarily large finite compactifications. \qed
\end{mylem}


The following example of the Cantor cube $2^\kappa$ for uncountable $\kappa$ shows that the assumptions in Lemma \ref{lemm2} cannot be considerably weakened. Since $\beta (2^\kappa \setminus \singleton{x})=2^\kappa$ \cite[Thm.\ 2]{Glicksberg}, these spaces have a unique compactification. 

The cube $2^\kappa$ is a zero-dimensional compact Hausdorff space with $2^\kappa \cong 2^\kappa  \oplus 2^\kappa$. For property $(\star)$, let $A \subset 2^\kappa \setminus \singleton{x}$ be a clopen non-compact subset. Since $2^\kappa \setminus \singleton{x}$ does not have a 2-point compactification, $A \cup \singleton{x}$ must be clopen in $2^\kappa$. But every clopen set of $2^\kappa$ can be written 
as a disjoint union of finitely many product-basic open sets, which are homeomorphic to $2^\kappa$. Hence $A \cup \singleton{x} \cong 2^\kappa$. 

We conclude that Lemma \ref{theorem1} applies, but restricts to the obvious assertion that the one-point compactification of $2^\kappa \setminus \singleton{x}$ is homeomorphic to $2^\kappa$.


\section{The space $\wstar$}
\label{section3}

This section contains the proof that under [CH] all finite compactifications of $\wstar \setminus \singleton{x}$ are homeomorphic to $\wstar$. The plan for attack is clear: we want to apply Lemma \ref{theorem1} to $\wstar$. That we may do so will be justified by Lemma \ref{lemma123}. Before that, we recall some characteristics of the space $\wstar$. 
 
Parovi\v{c}enko's theorem says that under [CH], the Stone-\v{C}ech remainder $\wstar$ of the countable discrete space is topologically characterised as the unique compact zero-dimensional F-space of weight $\cont$ without isolated points in which each non-empty $G_\delta$-set has non-empty interior \cite[1.2.4]{Intro}. 

Recall that an $F$-space is a space where all cozero sets are $\Cstar$-embedded. In normal spaces, the $F$-space property is equivalent to pairs of disjoint cozero sets having disjoint closure, and is therefore closed hereditary \cite[1.2.2]{Intro}. Also, it is known that infinite closed subsets of compact $F$-space contain a copy of $\beta \w$, and therefore have large cardinality. In particular, $\wstar$ does not contain converging sequences.

A $P$\emph{-point} is a point $p$ such that every countable intersection of neighbourhoods of $p$ contains again an open neighbourhood of $p$. In a zero-dimensional space, a point $x$ is a not a $P$-point if and only if there exists an open $F_\sigma$-set containing $x$ in its boundary.

\begin{mylem}
\label{lemma123}
\textnormal{[CH].} The space $\wstar$ has property $(\star)$, i.e.\ the one-point compactification of a clopen non-compact subset of $\wstar \setminus \singleton{x}$ is homeomorphic to $\wstar$.
\end{mylem}

\begin{proof}
Let $A$ be a clopen non-compact subset of $\wstar \setminus \singleton{x}$. Taking $A \cup \singleton{x}$, a closed subset of $\wstar$, as representative of its one-point compactification, we see that it is a zero-dimensional compact $F$-space of weight $\cont$ without isolated points. 

Suppose that $U \subset A \cup \singleton{x}$ is a non-empty $G_\delta$-set. If $U$ has empty intersection with $A$, then the singleton $U=\singleton{x}$ is a $G_\delta$-set, and hence has countable character  in the compact Hausdorff space $A \cup \singleton{x}$. It follows that there is a non-trivial sequence in $\wstar$ converging to $x$, a contradiction. Thus, $U$ intersects the open set $A$ and their intersection is a non-empty $G_\delta$-set of $\wstar$ with non-empty interior.

An application of Parovi\v{c}enko's theorem completes the proof.
\end{proof}

\begin{mythm}
\label{ClassificationCompactifications}
\textnormal{[CH].} Let $x$ be a point in $\wstar$. Every finite compactification of $\wstar \setminus \singleton{x}$ is homeomorphic to $\wstar$. Moreover, at most one point in the remainder of a finite compactification is not a $P$-point.
\end{mythm}

\begin{proof}
The first assertion follows immediately from Lemmas \ref{lemma123} and \ref{theorem1}.

For the second assertion, suppose there some remainder of $\wstar \setminus \singleton{x}$ contains two non-$P$-points $\infty_1$ and $\infty_2$. Let $A_1$ and $A_2$ be disjoint clopen neighbourhoods of $\infty_1$ and $\infty_2$ as in Theorem \ref{theorem1}, i.e.\ such that $A'_i = A_i \setminus \singleton{\infty_i}$ are clopen subsets of $\wstar \setminus \singleton{x}$. As the points at infinity are non-$P$-points, there are open $F_\sigma$-sets $F_1$ and $F_2$ with $F_i \subset A'_i$ such that $\infty_i \in \closure{F}_i \setminus F_i$. However, since $A'_i \cup \singleton{x} \cong A'_i \cup \singleton{\infty_i}$, we see that in $\wstar$ the disjoint open $F_\sigma$-sets $F_1$ and $F_2$ both limit onto $x$, contradicting the $F$-space property. 
 \end{proof} 

By a well-known result of Fine \& Gillman using [CH], every space $\wstar \setminus \singleton{x}$ splits into complementary clopen non-compact sets for all points $x$ of $\wstar$ \cite{Gillmann}. This gives rise to a two-point compactification of $\wstar \setminus \singleton{x}$. Lemma \ref{lemm2}, with Theorem \ref{ClassificationCompactifications}, now shows that $\wstar \setminus \singleton{x}$ has arbitrarily large finite compactifications. This also implies the well-known result that under [CH] the space $\wstar$ contains $P$-points. 


One may ask what of this remains true in absence of [CH]. In the above proof, [CH] is needed only in applying Parovi\v{c}enko's theorem. In ZFC, therefore, any finite compactification of $\wstar \setminus \singleton{x}$ is a Parovi\v{c}enko space of weight $\cont$ such that at most one point in the remainder is not a $P$-point.

Is it consistent with ZFC that no Parovi\v{c}enko space of weight $\cont$ contains $P$-points?


\section{The space $S_\kappa$}
\label{section4}
The spaces $S_\kappa$ are the natural generalisation of the Parovi\v{c}enko space $\wstar$ under [CH] to larger cardinals $\kappa$ with the property $\kappa=\kappa^{<\kappa}$. 

In a zero-dimensional space $X$, the \emph{type} of an open subset $U$ of $X$ is the least cardinal number $\tau$ such that $U$ can be written as a union of $\tau$-many clopen subsets of $X$. A zero-dimensional space where open subsets of type less than $\kappa$ are $\Cstar$-embedded is called $F_\kappa$-\emph{space} \cite[Ch. 14]{Ultrafilters}. Note that in a zero-dimensional compact space the notions of $F$- and $F_{\w_1}$-space coincide. It is well-known that being an $F_\kappa$-space implies that disjoint open sets of type less than $\kappa$ have disjoint closure and that under normality, the implication reverses \cite[6.5]{Ultrafilters}.

A $\kappa$-\emph{Parovi\v{c}enko space} is a zero-dimensional compact $F_\kappa$-space of weight $\kappa^{<\kappa}$ without isolated points such that every non-empty intersection of less that $\kappa$-many open sets has non-empty interior. Under $\kappa = \kappa^{<\kappa}$ there is a up to homeomorphism unique $\kappa$-Parovi\v{c}enko space of weight $\kappa$, denoted by $S_\kappa$ \cite[Ch. 6]{Ultrafilters}. Assuming [CH], we have $S_{\w_1} \cong \wstar$.

A $P_\kappa$\emph{-point} is a point $p$ such that the intersection of less than $\kappa$-many neighbourhoods of $p$ contains again an open neighbourhood of $p$. By zero-dimensionality, a point $x \in S_\kappa$ is a not a $P_\kappa$-point if and only if there exists an open set of type less than $\kappa$ containing $x$ in its boundary. Again, a $P_{\w_1}$-point is simply a $P$-point.

In this section we generalise results from $S_{\w_1}=\wstar$ to general $S_\kappa$, assuming $\kappa=\kappa^{<\kappa}$ throughout. The challenge lies in the fact that Lemma \ref{lemma123} does not carry through without extra work. Indeed, Lemma \ref{lemma123} rested on two corner stones: that in normal spaces, the $F$-space property is closed-hereditary and that every infinite closed subset of $S_{\w_1}$ has the same cardinality as $S_{\w_1}$. Both assertions do not carry over to $S_\kappa$, as it contains a closed copy of $\beta \w$.

The following shows how to circumvent these obstacles. 

\begin{mylem}
\label{nicelemma}
Let $x$ be a point in $S_\kappa$. If $A$ is a clopen, non-compact subset of $S_\kappa \setminus \singleton{x}$ then its type in $S_\kappa$ equals $\kappa$. 
\end{mylem}

\begin{proof}
Suppose for a contradiction that there exists a clopen, non-compact subset $A$ of $S_\kappa \setminus \singleton{x}$ of $S_\kappa$-type $\tau < \kappa$. It suffices to consider uncountable $\kappa$. Find a representation 
$$A=\bigcup_{\alpha < \tau} A_\alpha$$ where all $A_\alpha$ are clopen subsets of $S_\kappa$. We claim that there is a collection $\Set{V_\alpha}_{\alpha<\tau}$ of pairwise disjoint clopen sets of $S_\kappa$ such that $V_\alpha \subset A \setminus \bigcup_{\beta < \alpha}A_\beta$ for all $\alpha < \tau$. 

We proceed by transfinite induction. Choose a clopen subset $V_0$ in the non-empty open set $A \setminus A_0$. 
Now consider $\alpha < \tau$ and suppose that $V_\beta$ have been defined for all $\beta < \alpha$. By \cite[14.5]{Ultrafilters}, the set $U_\alpha= \bigcup_{\beta < \alpha} A_\beta \cup V_\beta$ cannot be dense in $A$, and we may find a clopen set $V_\alpha$ in the interior of $A \setminus U_\alpha$. This completes the inductive construction.

Finally, let $f$ and $g$ be disjoint cofinal subsets of $\tau$. We define disjoint sets
$$V_f=\bigcup_{\alpha \in f} V_{\alpha} \quad \text{and} \quad V_g=\bigcup_{\alpha \in g} V_{\alpha}$$ and claim that both sets limit onto $x$, contradicting the $F_\kappa$-space property of $S_\kappa$. 

Suppose the claim was false. Then $\closure{V_f}$ is a subset of $A=\bigcup_{\alpha < \tau} A_\alpha$. By compactness, there is a finite set $F \subset \tau$ such that $\closure{V_f} \subset \bigcup_{\beta \in F} A_\beta$. 
But there are sets $V_\alpha$ with arbitrarily large index constituting to $V_f$, a contradiction.
\end{proof}

An interesting corollary of this is that for uncountable $\kappa$, the boundary of every open set in $S_\kappa$ of type less than $\kappa$ is infinite, and hence, as a closed subset, of cardinality at least $2^\cont$.

\begin{mylem}
\label{superlemma}
Let $x$ be a point in $S_\kappa$. If $A$ is a clopen, non-compact subset of $S_\kappa \setminus \singleton{x}$ then its one-point compactifications is homeomorphic to $S_\kappa$. 
\end{mylem}

\begin{proof}
Let $X$ be the closure of $A$ in $S_\kappa$, i.e.\ $X=A \cup \singleton{x} \subset S_\kappa$. Then $X$ is a compact zero-dimensional space of weight $\kappa$. We check for the remaining $\kappa$-Parovi\v{c}enko properties.

To show that $X$ has the $F_\kappa$-space property, let $U$ and $V$ be disjoint open sets of $X$ of type less than $\kappa$. By normality, it suffices to show that $U$ and $V$ have disjoint closure in $X$. Suppose that $x$ belongs to $U \cup V$. Assume $x \in U$, so that $x$ does not belong to the closure of $V$. The sets $U\cap A$ and $V\cap A$ are of $A$-type less than $\kappa$. And since $A$ is an $F_\kappa$-space by \cite[14.1]{Ultrafilters}, they have disjoint closure in $A$, and therefore in $X$. Next, suppose that $x$ does not belong to $U \cup V$. Then $U$ and $V$ are subsets of $A$, 
and consequently of $S_\kappa$-type less than $\kappa$. Thus, $U$ and $V$ have disjoint closures in $S_\kappa$, and hence in $X$. This establishes that $X$ is an $F_\kappa$-space.

To show that $X$ has the property that every non-empty intersection of less than $\kappa$-many clopen sets has non-empty interior, suppose that $U=\bigcap_{\alpha < \beta} U_\alpha$ is a non-empty set, $\beta < \kappa$ and all $U_\alpha$ are clopen subsets of $X=A \cup \singleton{x}$. If $U$ has empty intersection with $A$, then all $X \setminus U_\alpha$ are clopen subsets of $S_\kappa$. It follows that $A=\bigcup_{\alpha < \beta} X \setminus U_\alpha$ is a clopen non-compact subspace $S_\kappa \setminus \singleton{x}$ of type less than $\kappa$, contradicting Lemma \ref{nicelemma}. Thus, $U$ intersects $A$, and their intersection has non-empty interior in $S_\kappa$.
\end{proof}

\begin{mythm}
\label{ClassificationCompactifications2}
Let $x$ be a point in $S_\kappa$. Every finite compactification of $S_\kappa \setminus \singleton{x}$ is homeomorphic to $S_\kappa$. Moreover, at most one point in the remainder of a finite compactification is not a $P_\kappa$-point.
\end{mythm}

\begin{proof}
As in Theorem \ref{ClassificationCompactifications}.
 \end{proof} 
 
As in the case of $\wstar$, the spaces $S_\kappa \setminus \singleton{x}$ split into complementary clopen non-compact sets  \cite[14.2]{Ultrafilters} and therefore have arbitrarily large finite compactifications by Lemma \ref{lemm2}.
Again, we obtain as a corollary that $S_\kappa$ contains $P_\kappa$-points. 

	  
\end{document}